\documentclass[11pt, reqno]{amsart}

\pdfoutput=1

\usepackage[T1]{fontenc}
\usepackage[UKenglish]{babel}
\usepackage{enumitem}
\usepackage{amsmath, amssymb}
\usepackage{tikz-cd}
\usepackage{mathtools}
\usepackage{graphicx}
\usepackage{mathrsfs}
\usepackage{stmaryrd}
\usepackage{bm}
\usepackage{mathrsfs}
\usepackage[scr=euler,cal=dutchcal]{mathalfa}
\usepackage{stackengine}
\usepackage[hyphens]{url}
\usepackage[breaklinks=true]{hyperref}
\hypersetup{
    colorlinks = true,
    urlcolor = black,
    linkcolor= black,
    citecolor= black,
    filecolor= black,
}
\usepackage[normalem]{ulem}
\usepackage{nicematrix}
\usepackage{booktabs}
\usepackage{scalerel}
\usepackage[bbgreekl]{mathbbol}
\DeclareSymbolFontAlphabet{\mathbbm}{bbold}
\DeclareSymbolFontAlphabet{\mathbb}{AMSb}

\providecommand{\noopsort}[1]{}

\DeclareFontFamily{U}{mathb}{\hyphenchar\font45}
\DeclareFontShape{U}{mathb}{m}{n}{
      <5> <6> <7> <8> <9> <10> gen * mathb
      <10.95> mathb10 <12> <14.4> <17.28> <20.74> <24.88> mathb12
      }{}
\DeclareSymbolFont{mathb}{U}{mathb}{m}{n}
\DeclareMathSymbol{\boxslash}{2}{mathb}{"6D}

\DeclareRobustCommand\longtwoheadrightarrow
     {\relbar\joinrel\twoheadrightarrow}
\DeclareRobustCommand\longtwoheadleftarrow
     {\twoheadleftarrow\joinrel\relbar}
\DeclareRobustCommand\longrightarrowtail
     {\lhook\joinrel\relbar\joinrel\rightarrow}
     
\makeatletter
\newcommand{\we}{}
\DeclareRobustCommand{\we}{%
  \mathrel{\vphantom{\rightarrow}\mathpalette\circle@arrow\relax}%
}
\newcommand{\circle@arrow}[2]{%
  \m@th
  \ooalign{%
    \hidewidth$#1\circ\mkern1mu$\hidewidth\cr
    $#1\longrightarrow$\cr}%
}
\makeatother

\makeatletter
\newcommand{\tfib}{}
\DeclareRobustCommand{\tfib}{%
  \mathrel{\vphantom{\twoheadrightarrow}\mathpalette\circle@twoheadarrow\relax}%
}
\newcommand{\circle@twoheadarrow}[2]{%
  \m@th
  \ooalign{%
    \hidewidth$#1\circ\mkern1mu$\hidewidth\cr
    $#1\longtwoheadrightarrow$\cr}%
}

\newcommand{\tfibleft}{}
\DeclareRobustCommand{\tfibleft}{%
  \mathrel{\vphantom{\twoheadleftarrow}\mathpalette\circle@twoheadleftarrow\relax}%
}
\newcommand{\circle@twoheadleftarrow}[2]{%
  \m@th
  \ooalign{%
    \hidewidth$#1\circ$\hidewidth\cr
    $#1\longtwoheadleftarrow\mkern1mu$\cr}%
}
\makeatother

\makeatletter
\newcommand{\tcofib}{}
\DeclareRobustCommand{\tcofib}{%
  \mathrel{\vphantom{\rightarrowtail}\mathpalette\circle@rightarrowtail\relax}%
}
\newcommand{\circle@rightarrowtail}[2]{%
  \m@th
  \ooalign{%
    \hidewidth$#1\circ\mkern1mu$\hidewidth\cr
    $#1\longrightarrowtail$\cr}%
}
\makeatother
\makeatletter
\renewcommand{\tocsection}[3]{%
  \indentlabel{\@ifnotempty{#2}{\bfseries\ignorespaces#1 #2\quad}}\bfseries#3}
\renewcommand{\tocsubsection}[3]{%
  \indentlabel{\@ifnotempty{#2}{\ignorespaces#1 #2\quad}}#3}

\newcommand\@dotsep{4.5}
\def\@tocline#1#2#3#4#5#6#7{\relax
  \ifnum #1>\c@tocdepth 
  \else
    \par \addpenalty\@secpenalty\addvspace{#2}%
    \begingroup \hyphenpenalty\@M
    \@ifempty{#4}{%
      \@tempdima\csname r@tocindent\number#1\endcsname\relax
    }{%
      \@tempdima#4\relax
    }%
    \parindent\z@ \leftskip#3\relax \advance\leftskip\@tempdima\relax
    \rightskip\@pnumwidth plus1em \parfillskip-\@pnumwidth
    #5\leavevmode\hskip-\@tempdima{#6}\nobreak
    \leaders\hbox{$\m@th\mkern \@dotsep mu\hbox{.}\mkern \@dotsep mu$}\hfill
    \nobreak
    \hbox to\@pnumwidth{\@tocpagenum{\ifnum#1=1\fi#7}}\par
    \nobreak
    \endgroup
  \fi}
\AtBeginDocument{%
\expandafter\renewcommand\csname r@tocindent0\endcsname{0pt}
}
\def\l@subsection{\@tocline{2}{0pt}{2.5pc}{5pc}{}}
\makeatother

\theoremstyle{plain}
\newtheorem{theorem}{Theorem}[section]
\newtheorem{lemma}[theorem]{Lemma}
\newtheorem{proposition}[theorem]{Proposition}
\newtheorem{corollary}[theorem]{Corollary}

\newtheorem*{claim*}{Claim}

\theoremstyle{definition}
\newtheorem{definition}[theorem]{Definition}
\newtheorem{remark}[theorem]{Remark}

\newtheorem{notation}[theorem]{Notation}

\newcommand{\op}{\mathrm{op}} 
\newcommand{\colim}
{\operatornamewithlimits{colim}}

\DeclareMathOperator{\Lan}{Lan}
\newcommand{\llp}[1]{{}^{\scaleobj{0.7}{\boxslash}}#1} 
\newcommand{\rlp}[1]{#1^{\mkern1mu \scaleobj{0.7}{\boxslash}}} 
\newcommand{\An}{\mathrm{AnExt}} 
\DeclareMathOperator{\cHo}{Ho_c} 
\DeclareMathOperator{\Mod}{Mod} 
\DeclareMathOperator{\R}{\mathbf{R}} 
\DeclareMathOperator{\M}{\mathscr{M}} 

\newcommand{\down}{{\downarrow}} 
\newcommand{\up}{{\uparrow}} 

\newcommand{\N}{\mathbb{N}} 

\renewcommand{\o}{\overline}
\newcommand{\w}{\widehat}

\newcommand{\one}{\mathbf{1}} 
\newcommand{\two}{\mathbf{2}} 
\newcommand{\acc}{\frown} 
\newcommand{\E}{\mathscr{E}} 
\newcommand{\ML}{\mathrm{ML}} 
\newcommand{\emptyp}{\varnothing} 
\DeclareMathOperator{\depth}{\mathrm{depth}}

\newcommand{\cof}{\rightarrowtail} 
\newcommand{\fib}{\twoheadrightarrow} 
\newcommand{\into}{\hookrightarrow} 
\newcommand{\id}{\mathrm{id}} 
\newcommand{\Cof}{\mathscr{C}} 
\newcommand{\We}{\mathscr{W}} 
\newcommand{\Fib}{\mathscr{F}} 
\newcommand{\Rex}{\Rrightarrow^{\exists}} 

\renewcommand{\epsilon}{\varepsilon}
\renewcommand{\theta}{\vartheta}
\renewcommand{\phi}{\varphi}

\DeclareMathOperator{\Set}{\mathbf{Set}} 
\DeclareMathOperator{\C}{\mathscr{D}} 
\newcommand{\K}{\mathbb{K}} 
\newcommand{\T}{\mathbb{T}} 
\renewcommand{\P}{\mathbb{P}} 
\newcommand{\St}{\mathbb{S}} 
\newcommand{\om}{\bbomega} 

\title{A Model Category for Modal Logic}

\author{Luca Reggio}
\address{University College London, UK}
\email{l.reggio@ucl.ac.uk}

\thanks{Research supported by the EPSRC grant EP/V040944/1.}

\begin{document}

\maketitle

\vspace{-1em}
\begin{abstract} 
We define Quillen model structures on a family of presheaf toposes arising from tree unravellings of Kripke models, leading to a homotopy theory for modal logic. Modal preservation theorems and the Hennessy--Milner property are revisited from a homotopical perspective.
\end{abstract}

\tableofcontents

\vspace*{-2em}
\section{Introduction}
Many basic questions in (finite) model theory, and more generally in logic, can be phrased as \emph{model existence} problems: is there a model satisfying a collection of (logically) specified properties? For instance, consider the problem of constructing an $\omega$-saturated elementary extension of a given structure~$M$. A possible solution goes as follows: start from $M$ and, using the Compactness Theorem for first-order logic, deduce the existence of an elementary extension $M_1$ of $M$ that is ``saturated relative to $M$''; repeat this process, now starting from $M_1$, to obtain an elementary chain
\[
M \subseteq M_1 \subseteq M_2 \subseteq \cdots
\]
whose union is an $\omega$-saturated elementary extension of $M$.

The approach based on the Compactness Theorem has some shortcomings: it is not applicable to finite model theory (since the Compactness Theorem fails over finite structures) and, more generally, it does not allow to properly handle logical resources, because we lose control over the combinatorial complexity of the objects involved. An alternative solution to the model existence problem is via the \emph{model construction} method, which consists in explicitly constructing a model satisfying the desired properties. In the example above, we could construct an $\omega$-saturated elementary extension as an appropriate ultrapower of $M$.

In general, an object obtained by means of a model construction argument is overdetermined, in the sense that it is specified up to isomorphism and not up to equivalence with respect to the prescribed logical properties. 
An analogous situation occurs in topology, when we are interested in studying a space up to homotopy equivalence, rather than up to homeomorphism.

The aim of this note is to make this analogy more precise in the case of modal logic, using the language of abstract homotopy theory (also referred to as \emph{categorical homotopy} or \emph{homotopical algebra}). To this end, after recalling the necessary preliminaries in \S \ref{s:preliminaries}, we define in \S\S \ref{s:tree-unravel}--\ref{s:model-structures} Quillen model structures on categories of presheaves closely related to (partial) tree unravellings of Kripke models. Arguably, this provides a flexible setting in which to manipulate objects up to logical equivalence and perform constructions that are invariant under such equivalences. To illustrate this idea, in \S \ref{s:Los-Tarski} we give a short proof of the {\L}o\'s-Tarski preservation theorem for modal logic and its relativisation to finite models~\cite{HNvB1998, Rosen1997}, and in \S \ref{s:HM} we offer a homotopical view on the Hennessy--Milner property.

\subsubsection*{To the (abstract) homotopy theorist}
The model categories we shall consider arise from Cisinski model structures on presheaf toposes. Thus, they are combinatorial model categories (i.e., locally presentable and cofibrantly generated), the cofibrations are the monomorphisms, and all objects are~cofibrant.

From a logical perspective, a central question is whether, in these model categories, two objects can be connected by a span of trivial fibrations. If the answer is positive, we say that they are \emph{Morita equivalent}. In view of Kenneth Brown's Factorization lemma, combined with Whitehead theorem, Morita equivalence between fibrant objects reduces to homotopy equivalence.

\subsubsection*{To the (modal) logician}
All results below can be extended to modal languages with any finite number of modal operators, but for simplicity we shall restrict ourselves to unimodal languages. 
The presheaves we are interested in are defined over categories of traces, i.e.\ Kripke models in which the reflexive transitive closure of the accessibility relation is a finite linear order.

Whereas p-morphisms (i.e.\ functional bisimulations) are the natural notion of morphism in modal logic, in the categorical setting it is convenient to consider arbitrary homomorphisms (i.e.\ functional simulations). 
A particularly useful class of homomorphisms is that of \emph{pathwise embeddings}, the homomorphisms that also reflect the validity of atomic formulas.

\section{Preliminaries}\label{s:preliminaries}

\subsection{Modal logic}
We recall the basic concepts of (propositional) modal logic; further notions will be introduced when necessary. For a comprehensive treatment of the subject, see e.g.~\cite{blackburn2002modal}.

The language of modal logic extends the language of classical propositional logic, based on propositional variables and Boolean connectives $\top$, $\bot$, $\wedge$, $\vee$ and $\neg$, by adding modalities $\Box$ and $\Diamond$. For each propositional variable $p$ in the language, we consider a corresponding unary relation symbol~$\mathscr{P}$. Intuitively, the interpretation of $\mathscr{P}$ in a structure describes the subset of elements in which $p$ is true (this viewpoint is based on the \emph{standard translation} of modal logic into first-order logic, see e.g.\ \cite[\S 2.4]{blackburn2002modal}). The structures in which modal formulas are interpreted are the Kripke models.

\begin{definition}[Kripke model]
A \emph{Kripke model} consists of 
\begin{enumerate}[label=(\roman*)]
\item a set $A$ equipped with a binary relation $\acc_{A}$, the \emph{accessibility relation}, 
\item a subset $\mathscr{P}^{A}\subseteq A$ for each unary relation $\mathscr{P}$ in the language, and 
\item a distinguished element $a\in A$. 
\end{enumerate}
A \emph{homomorphism of Kripke~models} 
\[
f\colon (A,\acc_{A},a) \to (B,\acc_{B},b)
\] 
is a function $f\colon A\to B$ such that the following conditions are satisfied:
\begin{enumerate}[label=(\roman*)]
\item For all $x,y\in A$, $x\acc_{A} y$ implies $f(x)\acc_{B}f(y)$.
\item For all unary relations $\mathscr{P}$ and all $x\in A$, $x\in \mathscr{P}^{A}$ implies $f(x)\in \mathscr{P}^{B}$.
\item $f(a)=b$.
\end{enumerate}
The category of Kripke models and their homomorphisms is denoted by $\K$.
\end{definition}

\begin{remark}
We can identify Kripke models as structures (in the usual model-theoretic sense) for a vocabulary containing a constant symbol, a binary relation, and a set of unary relations. A homomorphism of Kripke models then coincides with the usual notion of homomorphism between structures.
\end{remark}

We systematically drop the subscript from the accessibility relation $\acc_{A}$ and simply denote it by $\acc$. We also write $(A,a)$ instead of $(A,\acc,a)$, or even just $A$ if no confusion arises.

The interpretation of a modal formula $\phi$ in a Kripke model $(A,a)$ is defined by structural induction on $\phi$:

\begin{itemize}
\item If $\phi$ is a propositional variable $p$, then $(A,a)\models p$ just when $a\in \mathscr{P}^{A}$.
\item The semantics of Boolean connectives is the usual one, for example $(A,a)\models \phi_{1} \vee \phi_{2}$ just when $(A,a)\models \phi_{1}$ or $(A,a)\models \phi_{2}$.
\item If $\phi$ is of the form $\Box \psi$, then $(A,a)\models \phi$ just when $(A,x)\models \psi$ for \emph{all} $x\in A$ such that $a\acc x$.
\item If $\phi$ is of the form $\Diamond \psi$, then $(A,a)\models \phi$ just when $(A,x)\models \psi$ for~\emph{some} $x\in A$ such that $a\acc x$.
\end{itemize}
Note that the formula $\Box\phi$ is logically equivalent to $\neg \Diamond\neg\phi$.

Next, we introduce modal depth, a complexity measure of modal formulas. Intuitively, it measures the maximum nesting of modalities in a formula.
\begin{definition}[Modal depth]
The \emph{(modal) depth} of a modal formula~$\phi$, written $\depth(\phi)$, is defined inductively as follows:
\begin{itemize}
\item If $\phi$ is either $\top,\bot$ or a propositional variable $p$, $\depth(\phi)=0$.
\item $\depth(\neg\psi)=\depth(\psi)$.
\item $\depth(\phi_{1}\vee \phi_{2})=\max(\depth(\phi_{1}),\depth(\phi_{2}))=\depth(\phi_{1}\wedge \phi_{2})$.
\item $\depth(\Box\psi) = \depth(\psi)+1 = \depth(\Diamond\psi)$.
\end{itemize}
\end{definition}

\subsection{Model categories}

In this section we recall the notions of \emph{Quillen model category} and \emph{Quillen adjunction}. In fact, we will only deal with a special class of model categories, namely \emph{Cisinski model categories}, but the latter concept is not needed to understand the content of \S\ref{s:tree-unravel}--\ref{s:HM} and its definition is therefore deferred to Appendix~\ref{app:model-str-proofs}. For a more thorough introduction to Quillen and Cisinski model categories, see e.g.~\cite{Cisinski2019}.

A morphism $f$ in a category $\C$ has the \emph{left lifting property} with respect to a morphism $g$, and $g$ has the \emph{right lifting property} with respect to $f$, provided that every commutative square as displayed below admits a diagonal filler,
\[\begin{tikzcd}
\cdot \arrow{d}[swap]{f} \arrow{r} & \cdot \arrow{d}{g} \\
\cdot \arrow{r} \arrow[dashed]{ur}[description]{d} & \cdot
\end{tikzcd}\]
that is an arrow $d$ making the two triangles commute.

Given a class $\mathscr{M}$ of morphisms in $\C$, we denote by $\llp{\mathscr{M}}$ (respectively, by $\rlp{\mathscr{M}}$) the class of arrows having the left (respectively, right) lifting property with respect to all morphisms in $\mathscr{M}$.

\begin{definition}[Weak factorisation system]
A pair $(\mathscr{L},\mathscr{R})$ of classes of morphisms in a category $\C$ is a \emph{weak factorisation system} if it satisfies the following two conditions:
\begin{enumerate}[label=(\roman*)]
\item\label{weak-decomposition} Every arrow of $\C$ can be factored as an arrow of $\mathscr{L}$ followed by an arrow of $\mathscr{R}$.
\item\label{weak-determin} $\mathscr{L} = \llp{\mathscr{R}}$ and $\mathscr{R} = \rlp{\mathscr{L}}$.
\end{enumerate}
\end{definition}
Note that, in view of item~\ref{weak-determin} above, any of the two classes forming a weak factorisation system determines the other.

\begin{definition}[Model category]
A \emph{(Quillen) model structure} on a locally small and finitely bicomplete category $\C$ is a triple $(\We,\Cof,\Fib)$ of classes of morphisms in $\C$ satisfying the following two properties:
\begin{enumerate}[label=(\roman*)]
\item The class $\We$ has the two-out-of-three property. That is, given composable morphisms $u\colon X\to Y$ and $v\colon Y\to Z$, if any two of $u, v$, and $vu$ are in $\We$, so is the third.
\item $(\Cof\cap \We,\Fib)$ and $(\Cof,\Fib\cap\We)$ are weak factorisation systems.
\end{enumerate} 
A \emph{(Quillen) model category} is a category $\C$ equipped with a model structure $(\We,\Cof,\Fib)$.
\end{definition}
If $(\We,\Cof,\Fib)$ is a model structure, we refer to the arrows in each of the three classes as the \emph{weak equivalences}, \emph{cofibrations} and \emph{fibrations}, respectively. The arrows in $\Cof\cap \We$ are called \emph{trivial cofibrations}, and those in $\Fib\cap\We$ \emph{trivial fibrations}. Cofibrations and fibrations are denoted, respectively, by $\cof$ and $\fib$, and weak equivalences by $\we$. E.g., $\tfib$ denotes a trivial fibration.

\begin{definition}[Fibrant object]
An object $X$ of a model category $\C$ is \emph{fibrant} if the unique morphism $!\colon X\to\one$ to the terminal object is a fibration.
\end{definition}

\begin{remark}\label{rem:overdetermined}
The data defining a model structure is overdetermined: any two of the classes $\We$, $\Cof$ and $\Fib$ determine the third. For example, an arrow is a weak equivalence just when it can be factored as a trivial cofibration followed by a trivial fibration.

In fact, by a result of Joyal, a model structure on a category is completely determined by its cofibrations together with its class of fibrant objects; see \cite[Proposition~E.1.10]{Joyal2008} or \cite[Theorem 15.3.1]{Riehl2014}.
\end{remark}

\begin{definition}[Quillen adjunction]
Let $\C$ and $\C'$ be Quillen model categories. A \emph{Quillen adjunction} is an adjoint pair of functors
\[\begin{tikzcd}[column sep=1.0em]
\C \arrow[bend left=35]{rr}[description]{G} & \text{\tiny{$\top$}} & \C' \arrow[bend left=35]{ll}[description]{F}
\end{tikzcd}\]
such that the left adjoint $F$ preserves cofibrations and the right adjoint $G$ preserves fibrations. A \emph{right Quillen functor} is a functor $G$ admitting a left adjoint $F$ such that the pair $(F,G)$ is a Quillen adjunction.
\end{definition}

Given an adjoint pair of functors $(F,G)$ between model categories, $F$ preserves cofibrations just when $G$ preserves trivial fibrations, and $F$ preserves trivial cofibrations just when $G$ preserves fibrations. In particular, $(F,G)$ is a Quillen adjunction if, and only if, $F$ preserves trivial cofibrations and $G$ preserves trivial fibrations; cf.\ e.g.\ \cite[Remark~2.3.8]{Cisinski2019}.

\section{Tree unravellings and the presheaf categories \texorpdfstring{$\w{\P_k}$}{Pk Hat}}
\label{s:tree-unravel}

\subsection{Unravellings}
We start by recalling the technique of \emph{unravelling}, of central importance in modal logic, which originates from~\cite{DL1959,Sahlqvist1975}.
Henceforth, $k$ denotes an ordinal that is either finite or equal to $\omega$. 

For any Kripke model $(A,a)$, we consider its \emph{tree unravelling to depth~$k$} starting from $a$, denoted by $R_k (A,a)$. When $k=\omega$, this is usually referred to as the \emph{full unravelling} of $A$. $R_k (A,a)$ is the Kripke model whose elements are non-empty sequences $[a_0, \ldots, a_j]$ of length at most $k$ of elements of $A$ such that:
\begin{enumerate}[label=(\roman*)]
\item $a_0=a$, and
\item $a_{\ell} \acc a_{\ell+1}$ for all $0\leq \ell<j$.
\end{enumerate}
The accessibility relation on $R_k (A,a)$ is defined as follows: for any two sequences $s,t\in R_k (A,a)$, $s\acc t$ holds just when $t$ is obtained by making one further step along $\acc$. That is, $s\acc t$ if, and only if, there exist $a_0, \ldots, a_n$ such that $s = [a_0,\ldots, a_{n-1}]$, $t = [a_0,\ldots, a_{n-1},a_n]$ and $a_{n-1}\acc a_n$. 
For each unary relation $\mathscr{P}$, its interpretation in $R_k (A,a)$ consists of the sequences whose last element is in $\mathscr{P}^{A}$. Finally, the distinguished element of $R_k (A,a)$ is the one-element sequence $[a]$.

Note that (the reflexive transitive closure of) the accessibility relation on $R_k(A,a)$ is a tree order of height at most $k$ whose least element, also called the \emph{root}, is the distinguished element $[a]$. That is, $R_k(A,a)$ is a so-called synchronization tree:

\begin{definition}[Synchronization tree]
A Kripke model $(C,c)$ is a \emph{synchronization tree} if, for all $x\in C$, there is a unique finite sequence $c_0\acc \cdots \acc c_n$ of elements of $C$ with $c_0=c$ and $c_n = x$. The full subcategory of $\K$ consisting of the synchronization trees is denoted by $\T$.
\end{definition}

Any synchronization tree $(C,c)$ carries a ``definable'' tree order, in the sense that the reflexive transitive closure of its accessibility relation is a tree order. For every $x\in C$, we write $\up x$ and $\down x$, respectively, for the up-set and down-set of $x$ in this~order.

Let $\T_k$ denote the full subcategory of $\K$ defined by the synchronization trees of height at most $k$ (in particular, $\T_{\omega}=\T$). The assignment $(A,a) \mapsto R_k(A,a)$ extends to a functor $R_k \colon \K \to \T_k$ sending a homomorphism of Kripke models $f\colon (A,a)\to (B,b)$ to
\[
R_k f\colon R_k(A,a)\to R_k(B,b), \ \ [a_0, \ldots, a_j] \mapsto [f(a_0), \ldots, f(a_j)].
\]

\begin{lemma}\label{l:coreflections-K-Tk}
The functor $R_k$ is right adjoint to the inclusion $\T_k \into \K$:
\[\begin{tikzcd}[column sep = 4em]
\T_k \arrow[hookrightarrow]{r}[yshift=3pt]{\top} & \K \arrow[bend right=50]{l}[description]{R_k}
\end{tikzcd}\]
In other words, $\T_k$ is a coreflective subcategory of $\K$. 
\end{lemma}

\begin{definition}[Pathwise embedding]
A homomorphism of Kripke models $f\colon A\to B$ is a \emph{pathwise embedding} if it reflects the unary relations, i.e.\ for all $x\in A$ and all unary relation symbols $\mathscr{P}$ in the language, $f(x)\in \mathscr{P}^B$ implies $x\in \mathscr{P}^A$. The wide subcategory of $\T_k$ defined by the pathwise embeddings is denoted by $\St_k$.
\end{definition}

The notion of pathwise embedding admits a logical reading. Recall that a modal formula is \emph{existential} if it does not contain the modality $\Box$ and all subformulas in the scope of a negation symbol are propositional variables. With this terminology, the next lemma follows e.g.\ from the axiomatic result in \cite[Proposition~6.18]{AR2023} (a concrete proof will appear in the forthcoming~\cite{ALR23}):
\begin{lemma}\label{l:pathwise-emb-logic}
The following statements are equivalent for all Kripke models $A,B\in \K$ and all finite ordinals $k$:
\begin{enumerate}[label=(\arabic*)]
\item There exists a pathwise embedding $R_k A\to R_k B$.
\item For all existential modal formulas $\phi$ of depth at most~$k$, $A\models \phi$ entails $B\models \phi$.
\end{enumerate}
\end{lemma}

Recall that an injective homomorphism between Kripke models is an \emph{embedding} (also called a \emph{strong} injective homomorphism) if it reflects both the unary relations and the accessibility relation.
Every embedding is a pathwise embedding, but the converse need not hold. 
An example is provided by the retraction $A + A \to A$, for any $A\in \K$, that is the identity on each summand. However, a pathwise embedding between synchronization trees is an embedding provided it is injective:
\begin{lemma}\label{l:emb-vs-inj}
The following statements are equivalent for any arrow $f$ in~$\St_k$:
\begin{enumerate}[label=(\arabic*)]
\item\label{i:S_k-monic} $f$ is a monomorphism.
\item\label{i:S_k-inj} $f$ is injective.
\item\label{i:S_k-emb} $f$ is an embedding.
\end{enumerate}
\end{lemma}
\begin{proof}
\ref{i:S_k-monic} $\Rightarrow$ \ref{i:S_k-inj} Suppose $f\colon A\to B$ is a monomorphism in $\St_k$ and $f(x)=f(y)$ for some $x,y\in A$. Since $f$ is a pathwise embedding, the substructures of $A$ with universes $\down x$ and $\down y$, respectively, are isomorphic; just observe that they are both isomorphic to the substructure of $B$ with universe $\down f(x) = \down f(y)$. Thus there are $C\in \St_k$ and embeddings $m,n\colon C\rightarrowtail A$ such that the image of $m$ is $\down x$, and the image of $n$ is $\down y$. Then $f(x)=f(y)$ implies $f\circ m = f\circ n$, and since $f$ is a monomorphism in $\St_k$ we get $m=n$. This shows that $x = y$, and so $f$ is injective.

\ref{i:S_k-inj} $\Rightarrow$ \ref{i:S_k-emb} It is enough to prove that every injective morphism ${f\colon A\to B}$ in $\T_k$ reflects the relation $\acc$. Suppose that $x,y\in A$ satisfy $f(x) \acc f(y)$, and let $n$ be the height of $f(x)$ in the tree order of $B$. If $x'$ is the unique element of height $n$ below $y$, then $x'\acc y$ entails $f(x')\acc f(y)$. But $f(x)$ is the unique element covered by $f(y)$, and so $f(x) = f(x')$. As $f$ is injective, we get $x=x'$ and so $x\acc y$.

\ref{i:S_k-emb} $\Rightarrow$ \ref{i:S_k-monic} Clear, since every embedding is injective.
\end{proof}

\subsection{Presheaves over paths}

In general, the categories $\St_k$ are not complete nor cocomplete. In fact, if the modal vocabulary contains at least one unary relation symbol, then $\St_k$ does not admit an initial object nor a terminal one. In this subsection, we will consider an embedding into a presheaf category over a forest order as a means to (co)complete $\St_k$.

The presheaves we are interested in are defined on categories of synchronization trees whose tree order is a finite chain. We shall refer to the latter as \emph{paths}. For simplicity, we shall work with a small skeleton of the class of paths (e.g., by considering only those paths whose universes are initial segments of the natural numbers). We write $\P_k$ for the (small) full subcategory of~$\St_k$ whose objects are paths in this skeleton. Note that $\P_{\omega}= \bigcup_{k<\omega}{\P_{k}}$.

\begin{remark}\label{rem:Pk-forest}
The category $\P_k$ is a poset, since there is at most one pathwise embedding between any two paths. In fact, every homomorphism of Kripke models whose domain is a path is injective, so the morphisms in $\P_k$ are precisely the embeddings by Lemma~\ref{l:emb-vs-inj}. Therefore, $\P_k$ is a forest order and it has height $k$.
\end{remark}

The reason for considering the category $\P_k$ is that the inclusion $\P_k \hookrightarrow \St_k$ is \emph{dense}. This is the content of the next proposition, which can be understood as exhibiting the presheaf category \[\w{\P_k}\coloneqq [\P_k^\op,\Set]\] as a (co)completion of the category $\St_k$.
\begin{proposition}\label{prop:P_k-dense}
The restricted Yoneda embedding 
\[
\St_k \to \w{\P_k}, \ \ A \mapsto \left( \St_k( - , A)\colon \P_k^\op \to \Set \right)
\] 
is full and faithful. 
\end{proposition}
\begin{proof}
This follows from a straightforward adaptation of the observation that the full subcategory of $\T_k$ defined by the paths is dense in $\T_k$. Just observe that pathwise embeddings correspond precisely to (mediating morphisms induced by) cocones of embeddings on diagrams of paths.
\end{proof}

\begin{remark}\label{rem:hom-embeddings}
For any $A\in \St_k$, the presheaf $\St_k( - , A)\colon \P_k^\op \to \Set$ sends a path $P$ to the set of embeddings $P\rightarrowtail A$.
\end{remark}

By virtue of Proposition~\ref{prop:P_k-dense}, we will identify $\St_k$, and thus also $\P_k$, with a full subcategory of $\w{\P_k}$. The idea, developed in the following sections, is that we can give a homotopy theoretic account of some fundamental constructions in modal logic by considering model structures on the presheaf categories $\w{\P_k}$.

\section{A model structure on \texorpdfstring{$\w{\P_k}$}{Pk Hat}}
\label{s:model-structures}
As in the previous section, we fix an arbitrary ordinal $k\leq \omega$. Further, we let $\Pi$ be the set of all embeddings between paths in $\St_k$. 
We shall now define a model structure on the presheaf category $\w{\P_k}$; in view of Remark~\ref{rem:overdetermined}, it suffices to specify what the cofibrations and the fibrant objects in this model structure are.
\begin{proposition}\label{p:model-structure}
The category $\w{\P_k}$ admits a model structure such that:
\begin{enumerate}[label=(\roman*)]
\item\label{i:gen-cof} The cofibrations are precisely the monomorphisms and they coincide with the class $\llp(\rlp{\Pi})$.
\item\label{i:fibrant} A presheaf $X\in \w{\P_k}$ is fibrant just when it satisfies the following property: For all arrows $i\colon P\to Q$ in $\Pi$, if there exists an arrow $Q\to X$ in $\w{\P_k}$ then every arrow $P\to X$ admits a lifting along $i$:
\[\begin{tikzcd}
P \arrow{d}[swap]{i} \arrow{r} & X \\
Q \arrow[dashed]{ur} & {}
\end{tikzcd}\]
\end{enumerate}
\end{proposition}

A proof of Proposition~\ref{p:model-structure} is offered in Appendix~\ref{app:model-str-proofs}; this is a straightforward application of the theory of Cisinski model categories~\cite{Cisinski2006}, but we defer it because an understanding of this machinery is not required in the following sections. Henceforth, we shall regard $\w{\P_k}$ as a model category with respect to the model structure defined in Proposition~\ref{p:model-structure}.

For applications in modal logic, it is useful to have a characterisation of the cofibrations and the trivial fibrations between objects of $\St_k$. To this end, we recall the notion of p-morphism between Kripke models:

\begin{definition}[p-morphism]
A pathwise embedding $f\colon A\to B$ between Kripke models is a \emph{p-morphism} if, for all $x\in A$ and $y\in B$ such that $f(x)\acc y$, there exists $x'\in A$ satisfying $x\acc x'$ and $f(x')=y$.
\end{definition}

Note that any p-morphism whose codomain is a synchronization tree is necessarily surjective, since each element in its codomain can be reached from the distinguished one. 

\begin{lemma}\label{l:charact-co-fib}
The following statements hold for any arrow $f\colon A\to B$ in $\St_k$:
\begin{enumerate}[label=(\alph*)]
\item\label{i:cofib-emb} $f$ is a cofibration if, and only if, it is an embedding of Kripke models.
\item\label{i:fib-p-mor} $f$ is a trivial fibration if, and only if, it is a (surjective) p-morphism.
\end{enumerate}
\end{lemma}
\begin{proof}
\ref{i:cofib-emb} By Proposition~\ref{p:model-structure}\ref{i:gen-cof}, we have to show that $f$ is an embedding of Kripke models just when it is monic in $\w{\P_k}$. If $f$ is monic in $\w{\P_k}$ then, a fortiori, it is monic in $\St_k$, and thus an embedding by Lemma~\ref{l:emb-vs-inj}. The converse holds because the embedding $\St_k\to \w{\P_k}$ preserves monomorphisms.

\ref{i:fib-p-mor} We must prove that $f$ is a p-morphism just when all commutative squares in $\St_k$ as displayed below, with $i\in\Pi$, admit a diagonal filler.
\[\begin{tikzcd}
P \arrow{r}{\alpha} \arrow{d}[swap]{i} & A \arrow{d}{f} \\
Q \arrow{r}{\beta}  & B
\end{tikzcd}\]

Suppose that $f$ is a p-morphism and $Q$ is an $m$-element chain 
\[
q_1\acc \cdots \acc q_m.
\]
We identify $P$ with the substructure of $Q$ with universe $\{q_1,\ldots, q_l\}$ for some $l\leq m$. If $l=m$, there is nothing to prove. Otherwise, since $f$ is a p-morphism, there is $x_{l+1}\in A$ such that $\alpha(q_l)\acc x_{l+1}$ and $f(x_{l+1}) = \beta(q_{l+1})$. Extend $\alpha\colon P\to A$ to the substructure of $Q$ with universe $\{q_1,\ldots, q_{l+1}\}$ by sending $q_{l+1}$ to $x_{l+1}$. Note that this extension is a homomorphism because $f$ is a pathwise embedding, and in fact an embedding because so is $\beta$. Applying the previous argument repeatedly, we obtain an embedding $Q\to A$ which is a diagonal filler for the square above.

For the converse direction, suppose that $f$ is a pathwise embedding and let $x\in A$ and $y\in B$ satisfy $f(x) \acc y$. Consider the following commutative square in $\K$, 
\[\begin{tikzcd}
\down x \arrow[hookrightarrow]{r} \arrow{d}[swap]{f_{\restriction \down x}} & A \arrow{d}{f} \\
\down y \arrow[hookrightarrow]{r}  & B
\end{tikzcd}\]
where $\down x$ denotes the substructure of $A$ on the set $\{x'\in A\mid x'\acc x\}$, and similarly for $\down y$. Since $f$ is a pathwise embedding, its restriction $f_{\restriction \down x}$ is an embedding. Hence, the square above lies in $\St_k$ and so it admits a diagonal filler $d\colon \down y\to A$. The element $x'\coloneqq d(y)$ satisfies $x\acc x'$ and ${f(x')=y}$, showing that $f$ is a p-morphism. 
\end{proof}

\subsection{Morita equivalence}
In this subsection we introduce the notion of \emph{Morita equivalence}, which will play a key role in the following. We hasten to point out that the same terminology is used with a different meaning in algebra (for rings inducing equivalent categories of modules) and categorical logic (to refer to theories with equivalent categories of models in any topos); instead, we employ the definition arising from the theory of Lie groupoids (see e.g.~\cite{Zhu2009}):
\begin{definition}[Morita equivalence] 
Let $X,Y$ be objects in a model category $\C$. A \emph{Morita equivalence} between $X$ and $Y$ is a span of trivial fibrations connecting the two objects:
\[ 
X \tfibleft \cdot \tfib Y
\]
We say that $X$ and $Y$ are \emph{Morita equivalent} if there exists a Morita equivalence between them.
\end{definition}

Morita equivalence between tree unravellings of Kripke models captures an important logical equivalence relation between Kripke models:
\begin{theorem}\label{t:p-morph-emb-logic}
The following statements are equivalent for all Kripke models $A,B$ and all finite ordinals $k$:
\begin{enumerate}[label=(\arabic*)]
\item\label{i:Morita-eq-span} $R_k A$ and $R_k B$ are Morita equivalent in $\w{\P_k}$.
\item\label{i:Morita-eq-logic} $A \equiv^{\ML}_k B$, i.e.\ for all modal formulas $\phi$ of depth at most~$k$, $A\models \phi$ if, and only if, $B\models \phi$.
\end{enumerate}
\end{theorem}
\begin{proof}
It is well known that $A \equiv^{\ML}_k B$ just when there exists a span of p-morphisms $R_k A\leftarrow \cdot \rightarrow R_k B$ in $\K$. In fact, we can slightly improve this characterisation to the effect that $A \equiv^{\ML}_k B$ just when there exists a span of p-morphisms $R_k A\leftarrow \cdot \rightarrow R_k B$ in the subcategory $\St_k$; this is the content of \cite[Theorem~10.13]{AS2021}.
Thus, in view of Lemma~\ref{l:charact-co-fib}, item~\ref{i:Morita-eq-logic} implies item~\ref{i:Morita-eq-span}. 

For the converse, it is enough to prove that if $R_k A$ and $R_k B$ are connected by a span of trivial fibrations in $\w{\P_k}$ then they are also connected by a span of p-morphisms in $\St_k$. This can deduced by applying \cite[Theorem~6.4]{AR2023} to $R_k A$ and $R_k B$.\footnote{Here, $R_k A$ and $R_k B$ are regarded as objects of the \emph{arboreal category} $\T_k$, cf.\ \cite[Examples~5.4]{AR2023}.} The idea is as follows: suppose we are given a span of trivial fibrations 
\[
R_k A \mathrel{\mathop{\tfibleft}^{\alpha}} X \mathrel{\mathop{\tfib}^{\beta}} R_k B
\]
in $\w{\P_k}$, and consider the set
\[
\mathscr{B}\coloneqq \{(\alpha_P(m),\beta_P(m)) \mid P\in \P_k, \ m\in X(P)\}.
\]
We shall identify the embedding $\alpha_P(m)$ with its image, which is a submodel of $R_k A$ of the form $\down x$ for some $x\in R_k A$; similarly for $\beta_P(m)$.
It suffices to show that $\mathscr{B}$ is a \emph{back-and-forth system} in the terminology of~\cite{AR2023}, meaning that it satisfies the following three properties:
\begin{enumerate}[label=(\roman*)]
\item\label{i:initial} If $a\in A$ and $b\in B$ are the distinguished elements, then $([a],[b])\in \mathscr{B}$.
\item\label{i:forth} If $(\down x, \down y)\in \mathscr{B}$ and $x\acc x'$ in $R_k A$, there exists $y'\in R_k B$ such that $y\acc y'$ and $(\down x', \down y') \in \mathscr{B}$.
\item\label{i:back} If $(\down x, \down y)\in \mathscr{B}$ and $y\acc y'$ in $R_k B$, there exists $x'\in R_k A$ such that $x\acc x'$ and $(\down x', \down y') \in \mathscr{B}$.
\end{enumerate}

For item~\ref{i:initial}, if $m_0\colon [a]\to R_k A$ is the inclusion of the one-element submodel~$[a]$, and $\emptyp$ denotes the initial presheaf, we have a commutative square as follows.
\[\begin{tikzcd}
{\emptyp} \arrow{d}[swap]{!} \arrow{r}{!} & X \arrow{d}{\alpha} \\
{[a]} \arrow{r}{m_0} & R_k A
\end{tikzcd}\] 
Since $!\colon \emptyp \to [a]$ is a cofibration and $\alpha$ a trivial fibration, there is a diagonal filler $d\colon [a]\to X$. The composite $\beta\circ d\colon [a] \to R_k B$ is an embedding of Kripke models whose image is $[b]$. It follows in particular that the Kripke models $[a]$ and $[b]$ are isomorphic; if $P\in \P_k$ is a path isomorphic to either (and thus both) of $[a]$ and $[b]$, we see that
\[
([a],[b]) = (\alpha_P(d),\beta_P(d)) \in\mathscr{B}
\]
where, by the Yoneda Lemma, we have identified $d$ with an element of $X(P)$.

For item~\ref{i:forth}, assume that $(\down x, \down y)\in \mathscr{B}$ and $x\acc x'$, and let $P\in \P_k$ and $m\in X(P)$ be such that $(\down x, \down y) = (\alpha_P(m),\beta_P(m))$. There is a commutative square as displayed below,
\[\begin{tikzcd}
P \arrow{d}[swap]{i} \arrow{r}{m} & X \arrow{d}{\alpha} \\
{\down x'} \arrow{r} & R_k A
\end{tikzcd}\] 
where $i$ is the composite of the unique isomorphism $P\cong \down x$ with the inclusion $\down x\to \down x'$, and the bottom horizontal arrow is the inclusion map. Because $i$ is a cofibration and $\alpha$ a trivial fibration, there is a diagonal filler $d\colon \down x'\to X$. The image of $\beta\circ d\colon \down x'\to R_k B$ is of the form $\down y'$ for some $y'\in R_k B$, and $x\acc x'$ implies $y\acc y'$. The pair $(\down x', \down y')$ is easily seen to belong to $\mathscr{B}$.

Finally, item~\ref{i:back} is proved as item~\ref{i:forth}, inverting the roles of $A$ and $B$.
\end{proof}

The proof of Theorem~\ref{t:p-morph-emb-logic} utilises the notion of back-and-forth system from the axiomatic setting of \emph{arboreal categories}~\cite{AR2023}. When applied to the categories $\T_k$, for $k$ a finite ordinal, the ensuing notion of back-and-forth equivalence captures what is known in modal logic as the \emph{$k$-bisimilarity} relation. On the other hand, in the case of the category $\T_{\omega}=\T$, the arboreal notion of back-and-forth equivalence captures the usual \emph{bisimilarity} relation in modal logic, see e.g.\ \cite[\S 2.2]{blackburn2002modal}. Recall that two Kripke models are bisimilar just when they are connected by a span of p-morphisms in $\K$. Therefore, the same proof yields the following result:
\begin{theorem}\label{t:morita-eq-omega}
The following are equivalent for all Kripke models $A,B$:
\begin{enumerate}[label=(\arabic*)]
\item\label{i:Morita-eq-omega-span} $R_{\omega} A$ and $R_{\omega} B$ are Morita equivalent in $\w{\P_{\omega}}$.
\item\label{i:Morita-eq-omega-logic} $A$ and $B$ are bisimilar.
\end{enumerate}
\end{theorem}

Theorems~\ref{t:p-morph-emb-logic} and~\ref{t:morita-eq-omega} give a concrete description of Morita equivalence, in the categories of the form $\w{\P_k}$, between objects in the image of the functors~$R_k$. In Corollary~\ref{cor:mor-eq-omega} below we will see that, for the purpose of testing Morita equivalence, we can always work in the larger category $\w{\P_{\omega}}$.

\subsection{Resource indexing}
For all ordinals $m,n$ satisfying $m<n\leq \omega$, the inclusion functor $j_{m,n}\colon \P_m \hookrightarrow \P_n$ induces a ``restriction'' functor
\begin{equation}\label{eq:rho-m-n}
\rho_{m,n}\colon \w{\P_n} \to \w{\P_m}, \ \ X\mapsto X\circ j_{m,n}.
\end{equation}
The functor $\rho_{m,n}$ has a left adjoint $\lambda_{m,n}$ sending $Y\in \w{\P_m}$ to the left Kan extension 
\[
\lambda_{m,n} Y\coloneqq \Lan_{j_{m,n}}Y.
\] 
Explicitly, $\lambda_{m,n} Y$ is the extension of $Y$ obtained by setting 
\[
\lambda_{m,n} Y(P)=\emptyset
\] 
for all $P\in \P_n$ that do not belong to $\P_m$.\footnote{Just observe that the comma category $j_{m,n}/P$ is empty for all $P\in \P_n$ that do not belong to $\P_m$, where $j_{m,n}$ is considered as a functor $\P_m^\op \hookrightarrow \P_n^\op$.}

To improve readability, we write $\rho_m$ and $\lambda_m$ instead of $\rho_{m,m+1}$ and $\lambda_{m,m+1}$, respectively.
We thus have a chain of adjunctions as follows,
\begin{equation}\label{l:chain-adj-lambda-rho}
\begin{tikzcd}[column sep = 3em]
\w{\P_1} \arrow[bend left=50]{r}[description]{\lambda_1} & \w{\P_2} \arrow[l, "\rho_1", "\bot"' {yshift = 3pt}] \arrow[bend left=50]{r}[description]{\lambda_2} & \w{\P_3} \arrow[l, "\rho_2", "\bot"' {yshift = 3pt}] \arrow[bend left=50]{r}[description]{\lambda_3} & \phantom{\w{\P_4}} \arrow[l, "\rho_3", "\bot"' {yshift = 3pt}] \arrow[r, draw=none, "\hspace{-1em}\cdots\ \ \cdots" description] & {}
\end{tikzcd}
\end{equation}
along with adjunctions
\[\begin{tikzcd}[column sep = 3em]
\w{\P_m} \arrow[bend left=50]{r}[description]{\lambda_{m,\omega}} & \w{\P_{\omega}} \arrow[l, "\rho_{m,\omega}", "\bot"' {yshift = 3pt}] 
\end{tikzcd}\]
for all finite ordinals $m$. 

\begin{proposition}\label{p:Quillen-adj-mn}
The following statements hold for all ordinals $m,n$ satisfying $m<n\leq \omega$.
\begin{enumerate}[label=(\alph*)]
\item\label{i:Quillen-adju-lambda-rho} The adjoint pair $\lambda_{m,n}\dashv \rho_{m,n}$ is a Quillen adjunction.
\item\label{i:fully-faithful-lambda} $\lambda_{m,n}$ is fully faithful and preserves trivial fibrations.
\end{enumerate}
\end{proposition}
\begin{proof}
\ref{i:Quillen-adju-lambda-rho} It is easy to see that $\rho_{m,n}$ preserves trivial fibrations (it is this fact that will be used in Corollary~\ref{cor:mor-eq-omega} below); in fact, this is equivalent to saying that $\lambda_{m,n}$ preserves cofibrations (i.e., monomorphisms), which is immediate from the explicit description of $\lambda_{m,n}$. The proof that $\lambda_{m,n}$ preserves trivial cofibrations is more delicate and employs concepts defined in Appendix~\ref{app:model-str-proofs}, so we only sketch it. 
In view of \cite[Proposition~2.4.40]{Cisinski2019}, it is enough to show that $\lambda_{m,n}$ sends any arrow in the set $\Gamma$ defined in eq.~\eqref{eq:gamma-arrows} to a trivial cofibration in $\w{\P_n}$ ($\Gamma$ is a generating set for the class of anodyne extensions). The arrows in $\Gamma$ are induced by the universal property of specified pushout squares. Since $\lambda_{m,n}$ fixes these pushout squares, it sends arrows in $\Gamma$ to anodyne extensions in $\w{\P_n}$, which are in particular trivial cofibrations.

\ref{i:fully-faithful-lambda} Recall that, given any adjunction, the left adjoint is fully faithful precisely when the unit of the adjunction is a natural isomorphism. In turn, the unit of the adjunction $\lambda_{m,n}\dashv \rho_{m,n}$ is a natural isomorphism because the inclusion $j_{m,n}\colon \P_m \hookrightarrow \P_n$ is fully faithful; see e.g.\ \cite[Proposition~4.23]{Kelly1982}.
\end{proof}

For each finite ordinal $k$, since the functor $\lambda_{k,\omega}\colon \w{\P_k}\to \w{\P_\omega}$ is fully faithful by Proposition~\ref{p:Quillen-adj-mn}\ref{i:fully-faithful-lambda}, we shall identify $\w{\P_k}$ with a full subcategory of $\w{\P_\omega}$.

\begin{corollary}\label{cor:mor-eq-omega}
Let $k$ be a finite ordinal. Any two presheaves $X,Y\in \w{\P_k}$ are Morita equivalent in $\w{\P_k}$ if, and only if, they are Morita equivalent in $\w{\P_\omega}$.
\end{corollary}
\begin{proof}
Let us fix an arbitrary finite ordinal $k$ and presheaves $X,Y\in \w{\P_k}$. 

If $X$ and $Y$ are Morita equivalent in $\w{\P_k}$ then they are Morita equivalent in $\w{\P_\omega}$ because $\lambda_{k,\omega}\colon \w{\P_k}\to \w{\P_\omega}$ preserves trivial fibrations by Proposition~\ref{p:Quillen-adj-mn}\ref{i:fully-faithful-lambda}.

Conversely, assume that $X$ and $Y$ are Morita equivalent in $\w{\P_\omega}$. Since the right adjoint $\rho_{k,\omega}$ fixes $X$ and $Y$ (up to isomorphism), and preserves trivial fibrations because it is a right Quillen functor by Proposition~\ref{p:Quillen-adj-mn}\ref{i:Quillen-adju-lambda-rho}, any Morita equivalence ${X \tfibleft Z \tfib Y}$ in $\w{\P_\omega}$ yields a Morita equivalence ${X \tfibleft \rho_{k,\omega}Z \tfib Y}$ in~$\w{\P_k}$.
\end{proof}

\subsection{Factorisation-through}
It follows from the definition of model category that every morphism in $\w{\P_k}$ can be factored as a cofibration followed by a trivial fibration; note that this factorisation is far from being unique. The next result shows that, for morphisms in the full subcategory $\St_k$, we can find such a factorisation within $\St_k$.

\begin{proposition}\label{p:S-k-closed-under-fact}
Any morphism $f\colon A\to B$ in $\St_k$ can be factored as a cofibration followed by a trivial fibration
\[
A \cof X \tfib B
\] 
with $X\in \St_k$. Moreover, if $A,B$ are finite, we can assume that $X$ is finite.
\end{proposition}
\begin{proof}
Consider a morphism $f\colon A\to B$ in $\St_k$. We shall consider a new synchronization tree $A^\circ$ obtained by attaching to $A$ a copy of $\up fx$, for each $x\in A$, modulo the identification $x \sim fx$. More precisely, $A^\circ$ can be described as a wide pushout as follows. For each $y\in B$, denote by $\o{y}$ the submodel of $B$ consisting of those elements that are comparable with $y$. In other words, the universe of $\o{y}$ is $\up y \cup \down y$. For each $x\in A$, consider the pushout square
\[\begin{tikzcd}
\down x \arrow{d} \arrow{r} & \o{fx} \arrow{d}{\rho_x} \\
A \arrow{r}{\sigma_x} & A_x \arrow[ul, phantom, "\ulcorner", very near start]
\end{tikzcd}\]
in the category $\K$ (equivalently, in $\T_k$), where the left vertical map is the inclusion, and the top horizontal one is the restriction of $f$. Define $A^\circ$ as the wide pushout in $\K$ of the set of morphisms 
\begin{equation}\label{eq:sigma_x-diagram}
\{\sigma_x\colon A\to A_x\mid x\in A\}.
\end{equation}

Note that each $\sigma_x$ is an embedding, hence so is the induced morphism $j\colon A \to A^\circ$. Further, for each $x\in A$ there is an arrow $\tau_x\colon A_x\to B$ induced by the universal property of the pushout as displayed below.
\[\begin{tikzcd}[column sep = 3.5em]
\down x \arrow{d} \arrow{r} & \o{fx} \arrow{d}{\rho_x} \arrow[bend left = 25]{ddr} & \\
A \arrow{r}{\sigma_x} \arrow[bend right = 25]{drr}[description]{f} & A_x \arrow[ul, phantom, "\ulcorner", very near start] \arrow[dashed]{dr}[description]{\tau_x} & \\
& & B
\end{tikzcd}\]
The morphisms $\tau_x$, for $x\in A$, form a compatible cocone over the diagram in eq.~\eqref{eq:sigma_x-diagram}, hence they induce a (unique) mediating morphism $t\colon A^\circ \to B$. By construction, we have 
\[
f = t\circ j\colon A \to A^\circ \to B.
\] 
Since each element of $A^{\circ}$ lies in a substructure of the form $A_{x}$ for some $x\in A$, in order to show that $t$ is a pathwise embedding it is enough to prove that each $\tau_{x}$ is a pathwise embedding. In turn, this follows from the fact that both $\sigma_{x}$ and $\rho_{x}$ are embeddings.

We claim that $t$ is a p-morphism. 
Suppose that $w\in A^\circ$ and $y\in B$ satisfy $t(w)\acc y$. 
Assume for a moment that there exist $x\in A$ and $y'\in\o{fx}$ such that, identifying $A_x$ with a submodel of $A^\circ$, $\rho_x(y')=w$, $y\in\o{fx}$ and $y'\acc y$. Letting $x'\coloneqq \rho_x y \in A_x$, we get
\begin{gather*}
w = \rho_x(y') \acc \rho_x(y) = x' \\
t(x') = t \rho_x (y) = \tau_x \rho_x (y) = y,
\end{gather*}
showing that $t$ is a p-morphism. In order to show that $x$ and $y'$ satisfying these properties always exist, we distinguish two cases:
\begin{enumerate}[label=(\roman*)]
\item If $w$ belongs to the image of $j\colon A\to A^\circ$, let $x\in A$ satisfy $j(x)=w$. Then $y'\coloneqq fx$ satisfies
\begin{gather*}
\rho_x(y') = \rho_x fx = \sigma_x x = j(x)=w \\
y' = fx = tj(x) = t(w) \acc y,
\end{gather*}
so in particular $y\in\o{fx}$.
\item If $w$ does not belong to the image of $j$, there exist $x\in A$ and $y'\in \o{fx}$ such that, identifying $A_x$ with a submodel of $A^\circ$, $\rho_x(y')=w$.
Then
\[
y' = \tau_x \rho_x(y') = t(w) \acc y.
\]
Moreover, note that $y'\notin \down fx$ because $w$ does not belong to the image of $j$, hence $y'\in \up fx$. It follows that $y\in \up fx$ and thus $y\in\o{fx}$.
\end{enumerate}

Since the inclusion $\T_k\into\K$ creates colimits in view of Lemma~\ref{l:coreflections-K-Tk}, the Kripke model $A^\circ$ belongs to $\T_k$, and is finite whenever $A, B$ are finite (for, in that case, the diagram in eq.~\eqref{eq:sigma_x-diagram} is a finite diagram of finite Kripke models). 
Further, when regarded as morphisms of $\St_{k}$, it follows from Lemma~\ref{l:charact-co-fib} that $j\colon A\to A^\circ$ is a cofibration and $t\colon A^\circ \to B$ is a trivial fibration. 
\end{proof}

The proof of Proposition~\ref{p:S-k-closed-under-fact} is an example of how the combinatorial and homotopical approaches co-inform each other: it shows that a somewhat standard model construction has an inherent homotopical nature.

\section{A preservation theorem for modal logic}
\label{s:Los-Tarski}

Recall that a first-order sentence $\phi$ is \emph{preserved under extensions} if, for any structure $A$, and any induced substructure $B$ of $A$,
\[
B\models \phi \ \Longrightarrow \ A\models \phi.
\]
Note that $\phi$ is preserved under extensions just when, for all embeddings between structures $A'\rightarrowtail A$, if $A'\models \phi$ then $A\models \phi$; in the following, we shall tacitly make use of (the modal version of) this equivalent characterisation.

The {\L}o\'s-Tarski preservation theorem for first-order logic states that a first-order sentence is preserved under extensions if, and only if, it is equivalent to an existential sentence~\cite{Los1955,Tarski1955}. Recall that a first-order sentence is said to be \emph{existential} if it does not contain universal quantifiers and all subformulas in the scope of a negation symbol are atomic formulas.

A modal version of the {\L}o\'s-Tarski preservation theorem, to the effect that a modal formula is preserved under extensions of Kripke models just when it is equivalent to an existential modal formula, was proved by van Benthem; cf.~\cite{HNvB1998}. In analogy with the case of first-order logic, a modal formula is \emph{existential} if it does not contain the modality $\Box$ and all subformulas in the scope of a negation symbol are propositional variables.

An alternative proof of the modal preservation theorem was offered by Rosen~\cite{Rosen1997}. The advantage of Rosen's approach is twofold: it applies uniformly both to the class of finite Kripke models and to the class of all Kripke models, and in the process it shows that the depth of formulas (i.e., the maximum number of nested modalities) is preserved. 

We give a proof of Rosen's result based on the homotopical approach introduced in \S\ref{s:model-structures}.

\begin{theorem}\label{th:existential-hpt}
A modal formula of depth $\leq k$ is preserved under extensions of Kripke models if, and only if, it is equivalent to an existential modal formula of depth $\leq k$. Further, the result relativises to finite Kripke models.
\end{theorem}
\begin{proof}
It is easy to see that an existential modal formula is preserved under extensions of (finite) Kripke models.

To prove the converse, we rephrase the first part of the statement. Assume that $\phi$ is a modal formula of depth~$k$ preserved under extensions, and let $\Mod(\phi)$ be the full subcategory of $\K$ consisting of the Kripke models that satisfy $\phi$. Let $\Rex_k$ be the preorder on Kripke models defined as follows: $A\Rex_k B$ if, and only if, $B$ satisfies all modal formulas of depth $\leq k$ satisfied by $A$. A standard argument shows that $\phi$ is equivalent to an existential modal formula of depth $\leq k$ just when $\Mod(\phi)$ is upwards closed with respect to~$\Rex_k$. Suppose that $A\in\Mod(\phi)$ and $A\Rex_k B$. Since $A$ and $R_k A$ satisfy the same model formulas of depth $\leq k$, and similarly for $B$ and $R_k B$, it is enough to prove that $R_k A\in \Mod(\phi)$ entails $R_k B\in \Mod(\phi)$. 

By Lemma~\ref{l:pathwise-emb-logic}, the condition $A\Rex_k B$ is equivalent to the existence of a pathwise embedding $f\colon R_k A\to R_k B$.
In view of Proposition~\ref{p:S-k-closed-under-fact}, $f$ can be decomposed into a cofibration followed by a trivial fibration 
\[
R_k A \cof  X \tfib  R_k B
\]
in the Quillen model category $\w{\P_k}$, with $X\in\St_k$. By Lemma~\ref{l:charact-co-fib}, $R_k A \cof  X$ is an embedding and $X \tfib  R_k B$ is a p-morphism. Since $\phi$ is preserved under embeddings, and p-morphisms preserve (and reflect) the validity of modal formulas, we conclude that $R_k A\in \Mod(\phi)$ implies $R_k B\in \Mod(\phi)$. 

To see that the result holds also relative to finite Kripke models, note that $R_k A$ and $R_k B$ are finite whenever $A$ and $B$ are finite (because $k$ is a finite ordinal), and in that case we can assume that $X$ is finite by Proposition~\ref{p:S-k-closed-under-fact}.
\end{proof}

Let us say that a class of synchronization trees $\C\subseteq \St_{k}$ has the \emph{factorisation-through} property if every arrow $A\to B$ in $\St_{k}$ with $A,B\in\C$ admits a factorisation as a cofibration followed by a trivial fibration
\[
A \cof X \tfib B
\] 
such that $X\in \C$. With this terminology, Proposition~\ref{p:S-k-closed-under-fact} states that $\St_{k}$ has the factorisation-through property, and similarly for its subclass consisting of the finite synchronization trees. 

In fact, the proof of Proposition~\ref{p:S-k-closed-under-fact} shows that $\C$ has the factorisation-through property whenever it is closed in~$\St_{k}$ (equivalently, in $\K$) under taking submodels and colimits; moreover, if all members of $\C$ are finite, it suffices to assume that $\C$ is closed under taking submodels and \emph{finite} colimits. Therefore, the following is a consequence of (the proof of) Theorem~\ref{th:existential-hpt}:
\begin{corollary}
Let $\C\subseteq \K$ be a class of (finite) Kripke models closed under submodels and (finite) colimits. If $A\in\C$ entails $R_{k}A\in \C$ for all finite ordinals $k$, then Theorem~\ref{th:existential-hpt} relativises to $\C$.
\end{corollary}
\begin{proof}
Fix an arbitrary class $\C$ of (finite) Kripke models closed under (finite) colimits and such that $R_{k}$ restricts to $\C$, for all finite ordinals $k$. Reasoning as in the proof of Theorem~\ref{th:existential-hpt}, it suffices to show that any pathwise embedding $R_{k}A \to R_{k}B$ with $A,B\in\C$ can be factored as 
\[
R_k A \cof  X \tfib  R_k B
\]
with $X\in\St_k\cap \C$. In turn, this follows from the fact that $R_{k}A$ and $R_{k}B$ belong to $\St_k\cap \C$, which has the factorisation-through property because it is closed under submodels and (finite) colimits.
\end{proof}

For example, it follows from the previous corollary that Theorem~\ref{th:existential-hpt} relativises to the class of all (finite) synchronization trees of height at most $k$, for any $k\leq \omega$.

\section{The Hennessy--Milner property}
\label{s:HM}

We shall write $A\equiv^{\ML} B$ to indicate that Kripke models $A$ and $B$ have the same modal theory, i.e.\ $A\models \phi$ if and only if $B\models \phi$ for all modal formulas~$\phi$.
Any two bisimilar Kripke models have the same modal theory, but the converse is not true. The Hennessy--Milner property holds for a class of Kripke models if the restriction of the equivalence relation~$\equiv^{\ML}$ to that class coincides with the bisimilarity relation:
\begin{definition}[Hennessy--Milner property]
A class $\mathscr{H}$ of Kripke models has the \emph{Hennessy--Milner property} if any two $A,B\in\mathscr{H}$ satisfying $A\equiv^{\ML} B$ are bisimilar.
\end{definition}

In this section we will see that the Hennessy--Milner property can be understood, from a homotopical standpoint, as stating the existence of global sections for certain presheaves of Morita equivalences (see Theorem~\ref{th:HM-global-point-equiv}).

For any $A\in\K$, consider the chain of inclusions of partial unravellings
\[
\R A \ \ \coloneqq \ \ R_1 A \to R_2 A \to R_3 A \to \cdots
\]
We denote by $\om$ the ``ordinal category'' generated by the graph 
\[
1 \to 2 \to 3 \to \cdots
\]
and regard $\R A$ as an object of the functor category $[\om, \w{\P_{\omega}}]$. In order to give a homotopical account of the Hennessy--Milner property, we will consider a model structure on $[\om, \w{\P_{\omega}}]$ induced by that of $\w{\P_{\omega}}$; this is known as the \emph{projective model structure}.\footnote{In this case, the projective model structure coincides with the \emph{Reedy model structure} obtained by viewing $\om$ as a Reedy category, cf.\ e.g.\ \cite[Exercise~14.2.9]{Riehl2014}.}

For the next result, let us say that an arrow $\alpha\colon F\to G$ in $[\om, \w{\P_{\omega}}]$ is a \emph{levelwise fibration} if $\alpha_n\colon Fn \to Gn$ is a fibration in $\w{\P_{\omega}}$ for all $n\in\om$; levelwise cofibrations and weak equivalences are defined in a similar fashion.

\begin{proposition}\label{p:reedy-chains}
The category $[\om, \w{\P_{\omega}}]$ admits a model structure in which the fibrations and weak equivalences are, respectively, the levelwise fibrations and levelwise weak equivalences.
\end{proposition}
\begin{proof}
This follows e.g.\ from \cite[Theorem~11.6.1]{Hirschhorn2003}, using the fact that $\w{\P_{\omega}}$ is a cofibrantly generated model category (cf.\ Appendix~\ref{app:model-str-proofs}).
\end{proof}

\begin{remark}
In general, the cofibrations in the model structure on $[\om, \w{\P_{\omega}}]$ described in Proposition~\ref{p:reedy-chains} are not the levelwise cofibrations. However, they can be characterised, as in any model category, as those arrows with the left lifting property with respect to all trivial fibrations.  
\end{remark}

\begin{proposition}\label{p:unnatural-Morita-eq}
The following statements are equivalent for any two Kripke models $A$ and $B$:
\begin{enumerate}[label=(\arabic*)]
\item\label{i:HM} $A\equiv^{\ML} B$ implies that $A$ and $B$ are bisimilar.
\item\label{i:unnatur-natur-Morita} $\R A$ and $\R B$ are Morita equivalent whenever they are levelwise Morita equivalent.
\end{enumerate}
\end{proposition}
\begin{proof}
By Theorem~\ref{t:p-morph-emb-logic}, the diagrams $\R A$ and $\R B$ are levelwise Morita equivalent just when $A\equiv^{\ML}_n B$ for all $n\in\N$; in turn, this is equivalent to $A\equiv^{\ML} B$. Further, any Kripke model $C$ is bisimilar to its full unravelling $R_\omega C$. Hence, it suffices to show that $R_\omega A$ and $R_\omega B$ are bisimilar precisely when $\R A$ and $\R B$ are Morita equivalent in $[\om, \w{\P_{\omega}}]$.

Note that $\colim{\R C}\cong R_\omega C$ for any Kripke model $C$. Thus, by Theorem~\ref{t:morita-eq-omega}, it is enough to show that $\R A$ and $\R B$ are Morita equivalent in $[\om, \w{\P_{\omega}}]$ just when $\colim{\R A}$ and $\colim{\R B}$ are Morita equivalent in $\w{\P_{\omega}}$. Any Morita equivalence between the colimits yields a Morita equivalence between the diagrams, by restricting at all finite levels. Conversely, recall from Proposition~\ref{p:model-structure}\ref{i:gen-cof} that the generating set $\Pi$ of cofibrations in $\w{\P_{\omega}}$ consists of arrows between finitely presentable objects. It follows that, if $F,G\in [\om, \w{\P_{\omega}}]$ are Morita equivalent diagrams consisting of cofibrations in $\w{\P_{\omega}}$, their colimits are also Morita equivalent;\footnote{In fact, these categorical colimits are also \emph{homotopy colimits}.} see e.g.\ \cite[Lemma~7.4.1]{Hovey99}. In particular, this applies to the diagrams $\R A$ and $\R B$, thus concluding the proof.
\end{proof}

Given two objects $X,Y \in \w{\P_{\omega}}$, we would like to consider the ``presheaf of Morita equivalences'' between $X$ and $Y$. In general, there may be a proper class of Morita equivalences between any two objects, hence we restrict our attention to a small set of such equivalences:
\begin{lemma}\label{l:morita-eq-subobject}
Two presheaves $X,Y \in \w{\P_k}$ are Morita equivalent just when there are a subobject $W\to X\times Y$ and a Morita equivalence $X \tfibleft W \tfib Y$.
\end{lemma}
\begin{proof}
For the non-trivial direction, suppose that 
\[
X \mathrel{\mathop{\tfibleft}^{f}} Z \mathrel{\mathop{\tfib}^{g}} Y
\]
is a Morita equivalence in $\w{\P_k}$, and let $W$ be the image of ${\langle f, g \rangle\colon Z\to X\times Y}$. In other words, we consider the (epi, mono) factorisation of $\langle f, g \rangle$, as displayed below.
\[\begin{tikzcd}
Z \arrow{r}{e} \arrow[bend left = 40]{rr}[description]{\langle f, g \rangle} & W \arrow{r}{m} & X\times Y
\end{tikzcd}\]
We claim that the following is a Morita equivalence
\[
X \mathrel{\mathop{\tfibleft}^{\tilde{f}}} W \mathrel{\mathop{\tfib}^{\tilde{g}}} Y
\]
where $\tilde{f}$ and $\tilde{g}$ are, respectively, the composite of $m$ with the product projections $X\times Y \to X$ and $X\times Y \to Y$. We only show that $\tilde{f}$ is a trivial fibration; the proof for $\tilde{g}$ is the same, mutatis mutandis. It suffices to prove that $\tilde{f}$ has the right lifting property with respect to any $i\in \Pi$. Consider arrows $\alpha\colon P\to W$ and $\beta\colon Q\to X$ such that $\tilde{f}\circ \alpha = \beta\circ i$. Since representable presheaves are projective objects, there exists $\alpha'\colon P\to Z$ such that $e\circ \alpha' = \alpha$. We therefore get a commutative diagram as shown below.
\[\begin{tikzcd}
{} & {} & {} & Z \arrow{dd}{e} \\
{} & {} & {} & {} \\
{} & P \arrow{uurr}{\alpha'} \arrow{rr}{\alpha} \arrow{dl}[swap]{i} & {} & W \arrow{dl}{\tilde{f}} \\
Q \arrow{rr}{\beta}  & {} & X & {}
\ar[from=1-4,to=4-3, crossing over, "f", swap]
\end{tikzcd}\]
Because $f$ is a trivial fibration, there exists an arrow $d\colon Q\to Z$ such that $d\circ i = \alpha'$ and $f\circ d = \beta$. It follows that $d'\coloneqq e\circ d\colon Q \to W$ is a diagonal filler for the bottom square; just observe that 
\[
d'\circ i = e \circ d \circ i = e\circ \alpha' =\alpha
\]
and
\[
\tilde{f} \circ d' = \tilde{f} \circ e \circ d = f\circ d = \beta.\qedhere
\]
\end{proof}

Let $X$ and $Y$ be any two objects of $\w{\P_k}$. We write
\[
\M(X,Y)
\] 
for the set of all Morita equivalences between $X$ and $Y$ witnessed by a subobject of $X\times Y$. In other words, the elements of $\M(X,Y)$ are (equivalence classes of) jointly monic spans of trivial fibrations $X \tfibleft \cdot \tfib Y$. Note that $\M(X,Y)$ is a set, as opposed to a proper class, because $\w{\P_k}$ is well-powered.

Given Kripke models $A,B\in \K$, consider the functor
\begin{equation}\label{eq:presheaf-Morita-equiv}
\M^{A,B}\colon \om^{\op}\to \Set, \ \ \M^{A,B}(n)\coloneqq \M(R_n A, R_n B).
\end{equation}
On morphisms, $\M^{A,B}$ sends an arrow $m \to n$ to the function 
\[
\M(R_n A, R_n B) \to \M(R_m A, R_m B)
\] 
that sends a jointly monic span of trivial fibrations $R_n A \tfibleft Z \tfib R_n B$ to its image
\[
R_m A \tfibleft \rho_{m,n} Z \tfib R_m B
\]
under the functor $\rho_{m,n}\colon \w{\P_n} \to \w{\P_m}$ from eq.~\eqref{eq:rho-m-n}. Note that the latter span is jointly monic because $\rho_{m,n}$ is right adjoint, and consists of trivial fibrations in view of Proposition~\ref{p:Quillen-adj-mn}\ref{i:Quillen-adju-lambda-rho}.

Let us say that a presheaf $F\in [\om^\op,\Set]$ is \emph{strongly non-empty} provided that $Fn\neq\emptyset$ for all $n\in\om$. A \emph{global section} of $F$ is a morphism $\one\to F$, where $\one$ is the terminal object of $[\om^\op,\Set]$. With this terminology, we have the following consequence of Proposition~\ref{p:unnatural-Morita-eq} and Lemma~\ref{l:morita-eq-subobject}.

\begin{theorem}\label{th:HM-global-point-equiv}
The following statements are equivalent for any class $\mathscr{H}$ of Kripke models:
\begin{enumerate}
\item $\mathscr{H}$ has the Hennessy--Milner property.
\item For all $A,B\in \mathscr{H}$, if the presheaf $\M^{A,B}$ is strongly non-empty then it has a global section.
\end{enumerate}
\end{theorem}
\begin{proof}
Fix arbitrary Kripke models $A,B\in \mathscr{H}$ and note that, by Lemma~\ref{l:morita-eq-subobject}, the presheaf $\M^{A,B}$ is strongly non-empty just when $R_n A$ and $R_n B$ are Morita equivalent for all finite ordinals $n$. In turn, the latter condition is equivalent to $A\equiv^{\ML} B$ by Theorem~\ref{t:p-morph-emb-logic}. It remains to prove that $A$ and $B$ are bisimilar if, and only if, $\M^{A,B}$ admits a global section.

The global sections of $\M^{A,B}$ can be identified with the indexed sets 
\[
\{ (R_n A \mathrel{\mathop{\tfibleft}^{f_n}} W_n \mathrel{\mathop{\tfib}^{g_n}} R_n B) \in \M(R_n A, R_n B) \mid n\in\om \}
\]
such that $\rho_n(f_{n+1}) = f_n$ and $\rho_n(g_{n+1}) = g_n$ for all $n$, where the functor $\rho_{n}$ is as in eq.~\eqref{l:chain-adj-lambda-rho}. If $A$ and $B$ are bisimilar, then by Theorem~\ref{t:morita-eq-omega} and Lemma~\ref{l:morita-eq-subobject} there is a jointly monic span of trivial fibrations 
\[
R_{\omega} A \mathrel{\mathop{\tfibleft}^{f}} W \mathrel{\mathop{\tfib}^{g}} R_{\omega} B
\]
in $\w{\P_{\omega}}$, which induces a global section
\[
\{ (R_n A \mathrel{\mathop{\tfibleft}^{\rho_{n,\omega}(f)}} \rho_{n,\omega}W \mathrel{\mathop{\tfib}^{\rho_{n,\omega}(g)}} R_n B) \in \M(R_n A, R_n B) \mid n\in\om \}.
\]
Conversely, the existence of a global section implies that the diagrams $\R A$ and $\R B$ are Morita equivalent in $[\om, \w{\P_{\omega}}]$, and so their colimits $R_\omega A$ and $R_\omega B$ are bisimilar (cf.\ the proof of Proposition~\ref{p:unnatural-Morita-eq}). Hence, $A$ and $B$ are also bisimilar.
\end{proof}

The definition of the presheaf of Morita equivalences $\M^{A,B}\colon \om^{\op}\to \Set$ in eq.~\eqref{eq:presheaf-Morita-equiv}, where $A$ and $B$ are Kripke models, can be generalised in a straightforward manner to all objects of $\w{\P_{\omega}}$. To this end, for all $X, Y\in \w{\P_{\omega}}$, let 
\[
\M^{X,Y}\colon \om^{\op}\to \Set, \ \ \M^{X,Y}(n)\coloneqq \M(\rho_{n,\omega} X, \rho_{n,\omega} Y).
\]
On morphisms, $\M^{X,Y}$ sends an arrow $m \to n$ to the function 
\[
\M(\rho_{n,\omega} X,\rho_{n,\omega} Y) \to \M(\rho_{m,\omega} X, \rho_{m,\omega} Y)
\] 
that sends a jointly monic span of trivial fibrations $\rho_{n,\omega} X \tfibleft Z \tfib \rho_{n,\omega} Y$ to its image
\[
\rho_{m,\omega} X \tfibleft \rho_{m,n} Z \tfib \rho_{m,\omega} Y
\]
under the functor $\rho_{m,n}\colon \w{\P_n} \to \w{\P_m}$.

\begin{definition}[Homotopical Hennessy--Milner property]
A class $\mathscr{X}$ of objects of $\w{\P_{\omega}}$ has the \emph{homotopical Hennessy--Milner property} if, for all ${X, Y\in \mathscr{X}}$, the presheaf $\M^{X,Y}$ is strongly non-empty just when it has a global section.
\end{definition}

\begin{remark}\label{rem:homotop-Morita-equiv}
Equivalently, a class $\mathscr{X}$ of objects of $\w{\P_{\omega}}$ has the homotopical Hennessy--Milner property if any two ${X, Y\in \mathscr{X}}$ are Morita equivalent exactly when $\rho_{n,\omega} X$ and $\rho_{n,\omega} Y$ are Morita equivalent for all $n\in\om$.
\end{remark}

\begin{lemma}
The following statements are equivalent for any class $\mathscr{H}$ of Kripke models: 
\begin{enumerate}
\item $\mathscr{H}$ has the Hennessy--Milner property.
\item The class $\{R_{\omega} A \mid A\in \mathscr{H}\}\subseteq \w{\P_{\omega}}$ has the homotopical Hennessy--Milner property.
\end{enumerate}
\end{lemma}
\begin{proof}
This is an immediate consequence of Theorem~\ref{th:HM-global-point-equiv}.
\end{proof}

We conclude by observing that the homotopical Hennessy--Milner property holds for the class of fibrant objects.
\begin{proposition}
The class of fibrant objects of $\w{\P_{\omega}}$ has the homotopical Hennessy--Milner property.
\end{proposition}
\begin{proof}
Clearly, if a presheaf of the form $\M^{X,Y}$ has a global section then it is strongly non-empty.

Conversely, using the reformulation in Remark~\ref{rem:homotop-Morita-equiv}, we must show that any two fibrant objects $X,Y\in \w{\P_{\omega}}$ are Morita equivalent whenever $\rho_{n,\omega} X$ and $\rho_{n,\omega} Y$ are Morita equivalent for all $n\in\om$. We shall adopt a strategy similar to the one we used in the proof of Theorem~\ref{t:p-morph-emb-logic}, reducing the problem of constructing a span of trivial fibrations connecting~$X$ and~$Y$ to that of defining a back-and-forth system between them. 

This approach is based on \cite[Theorem~6.4]{AR2023}, and the reason why it applies in this context is that $\w{\P_{\omega}}$ is a category of presheaves over a forest order, hence an arboreal category \cite[Theorem~V.5]{RR}.
To make the connection with the lifting property characterising fibrant objects more explicit, instead of exhibiting a back-and-forth system between~$X$ and~$Y$, we shall equivalently prove that Duplicator has a winning strategy in a back-and-forth game played on $X$ and $Y$; see \cite[Theorem~6.12]{AR2023}.

The initial position of the game is given by the pair of unique morphisms 
\[
(\, !\colon \emptyp \to X, \ !\colon \emptyp \to Y \, ).
\]
In the first round, Spoiler chooses one of the objects, say $X$, and an arrow $m\colon P\to X$ where $P$ is a minimal element in the forest order $\P_{\omega}$. Duplicator must respond with an arrow $n\colon P\to Y$. Duplicator wins the first round if they are able to respond; in that case, the new position is $(m,n)$. In the following round, Spoiler chooses again one of the objects, say $Y$, and an arrow $n'\colon Q\to Y$ extending $n$ such that $Q$ covers $P$ in the order $\P_{\omega}$. Duplicator must respond with $m'\colon Q\to X$ extending $m$, and so forth. Duplicator wins the game if they have a strategy that allows them to play forever. 

It is not difficult to see that, if $X$ and $Y$ are fibrant, Duplicator has a winning strategy. At any stage of the game, the position is specified by a pair of morphisms 
\[
(\, m\colon P \to X, \ n\colon P \to Y \, ).
\]
Suppose that in the next round Spoiler chooses a morphism $m'\colon Q\to X$ extending $m$, with $Q\in \P_{n}$. We claim that there exists a morphism $Q\to Y$. By assumption, there is a Morita equivalence 
\[
\rho_{n,\omega} X \mathrel{\mathop{\tfibleft}^{f}} Z \mathrel{\mathop{\tfib}^{g}} \rho_{n,\omega} Y.
\]
The morphism $m'\colon Q\to X$ induces a morphism $\mu'\colon Q\to \rho_{n,\omega} X$, and using the fact that $f$ is a trivial fibration we get a morphism $Q\to Z$:
\[\begin{tikzcd}
\emptyp \arrow{d}[swap]{!} \arrow{r}{!} & Z \arrow{d}{f} \\
Q \arrow{r}{\mu'} \arrow[dashed]{ur} & \rho_{n,\omega} X
\end{tikzcd}\]
Composing the latter morphism with $g$, we obtain a morphism $Q\to \rho_{n,\omega} Y$, and thus also a morphism $Q\to Y$. Now, since $Y$ is fibrant, item~\ref{i:fibrant} in Proposition~\ref{p:model-structure} implies that there exists an arrow $n'\colon Q\to Y$ extending~$n$. Thus, Duplicator wins the round. Similarly, if Spoiler chooses a morphism into $Y$, then Duplicator can use the fact that $X$ is fibrant to win the round. This yields a Duplicator winning strategy for the game.
\end{proof}

\appendix
\section{A proof of Proposition~\ref{p:model-structure}}\label{app:model-str-proofs}
Recall that a model structure $(\We,\Cof,\Fib)$ on a Grothendieck topos $\E$ is a \emph{Cisinski model structure}~\cite{Cisinski2006} if it satisfies the following two properties:
\begin{enumerate}[label=(\roman*)]
\item it is cofibrantly generated, i.e.\ there exist small sets $I,J$ permitting the small object argument such that $\Cof = \llp{(\rlp{I})}$ and $\Cof \cap \We = \llp{(\rlp{J})}$.
\item $\Cof$ coincides with the class of monomorphisms.
\end{enumerate}
By Remark~\ref{rem:overdetermined}, a Cisinski model structure is completely determined by its class of fibrant objects. 

\begin{remark}
For a discussion of what it means to ``permit the small object argument'', see e.g.\ \cite[\S 12.2]{Riehl2014}. The conditions $\Cof = \llp{(\rlp{I})}$ and $\Cof \cap \We = \llp{(\rlp{J})}$ are typically expressed by saying that $I$ and $J$ \emph{generate}, respectively, the classes of cofibrations and of trivial cofibrations. 
\end{remark}

The model structure in Proposition~\ref{p:model-structure} can be obtained by means of a general procedure to construct Cisinski model structures on categories of presheaves on a small category; see~\cite{Cisinski2006} or \cite[\S 2.4]{Cisinski2019}.

Recall that $\Pi$ denotes the set of all morphisms in $\w{\P_k}$ whose domain and codomain are representable (note that $\Pi$ is a small set because $\P_k$ is a small category).

\begin{lemma}\label{l:Pi-generator}
The set $\Pi$ generates the class of monomorphisms of $\w{\P_k}$, i.e.\ $\llp{(\rlp{\Pi})}$ is the class of all monomorphisms.\footnote{In the terminology of \cite[Definition~2.4.4]{Cisinski2019}, $\Pi$ is a \emph{cellular model}.}
\end{lemma}
\begin{proof}
The class of monomorphisms of $\w{\P_k}$ is generated by the set $I$ consisting of all monomorphisms $X\to P$ with $P$ a quotient of a representable presheaf; see e.g.\ \cite[Example~2.1.11]{Cisinski2019}. We claim that $I = \Pi$. 

The inclusion $\Pi\subseteq I$ follows from Remark~\ref{rem:Pk-forest}, combined with the fact that the Yoneda embedding preserves monomorphisms. For the converse inclusion observe that, in $\w{\P_k}$, representable presheaves have no non-trivial quotients, and a subpresheaf of a representable presheaf is also representable.
\end{proof}

Next, consider the presheaf $\two\in\w{\P_k}$, defined as $\two\coloneqq \one+\one$, which comes equipped with coproduct maps $0,1 \colon \one\to\two$. This induces an endofunctor
\[
- \times \two \colon \w{\P_k} \to \w{\P_k}.
\] 
Since $X \times \two\cong X + X$ for all $X\in\w{\P_k}$, the coproduct injections of $X$ correspond to arrows $\partial_0,\partial_1\colon X \to X\times \two$.\footnote{This data defines an \emph{exact cylinder} in the sense of \cite[Definition~2.4.8]{Cisinski2019}; cf.~Example~2.4.10 in \emph{op.\ cit.}} 

\begin{notation}
When $X = \one$, we shall write $\partial_e\colon \{e\}\to \two$, for $e\in\{0,1\}$, to denote the coproduct injections.
\end{notation}

Let $f,g\colon X \to Y$ be arrows in $\w{\P_k}$. A \emph{homotopy} from $f$ to $g$, denoted by $f \Rightarrow g$, is a morphism $\eta\colon X\times \two \to Y$ making the following diagram commute.
\[\begin{tikzcd}
X \arrow[bend right = 30]{dr}[description]{f} \arrow{r}{\partial_0} & X \times \two \arrow{d}[description]{\eta} & X \arrow{l}[swap]{\partial_1} \arrow[bend left = 30]{dl}[description]{g} \\
 & Y & 
\end{tikzcd}\]
Since $X\times\two \cong X + X$, the universal property of the coproduct entails that between any two parallel arrows in $\w{\P_k}$ there is always a (unique) homotopy. In particular, the homotopy relation $\Rightarrow$ on hom-sets $\w{\P_k}(X,Y)$ is an equivalence relation. Denote by 
\[
[X,Y]
\]
the quotient of the set $\w{\P_k}(X,Y)$ with respect to this equivalence relation.

\begin{remark}\label{rem:cHo-preorder}
The \emph{classical homotopy category} of $\w{\P_k}$ is the category $\cHo(\w{\P_k})$ whose objects are those of $\w{\P_k}$, and for all $X,Y \in\cHo(\w{\P_k})$, 
\[
\cHo(\w{\P_k})(X,Y)\coloneqq [X,Y].
\] 
Note that $\cHo(\w{\P_k})$ is a (large) preorder, as a hom-set $[X,Y]$ is either empty or a one-element set (because any two parallel arrows in $\w{\P_k}$ are homotopic). Thus, the quotient functor $\w{\P_k} \to \cHo(\w{\P_k})$ is the poset reflection of $\w{\P_k}$.
\end{remark}

Let $\Gamma$ be the set of monomorphisms
\begin{equation}\label{eq:gamma-arrows}
(P \times \two) \cup (Q \times \{e\}) \to Q \times \two
\end{equation}	
induced by the universal property of the pushout, for $e\in \{0,1\}$ and ${i\colon P\to Q}$ in $\Pi$:
\[\begin{tikzcd}[column sep = 2.5em]
 P \times \{e\} \arrow{r}{\id_P\times \partial_e} \arrow{d}[swap]{i\times \id_{\{e\}}} & P \times \two \arrow{d} \arrow[bend left=30]{ddr}[description]{i\times\id_{\two}} & {} \\
 Q \times \{e\} \arrow{r} \arrow[bend right=30]{drr}[description]{\id_{Q}\times\partial_e} & (P \times \two) \cup (Q \times \{e\}) \arrow[dashed]{dr} \arrow[ul, phantom, "\ulcorner", very near start] & {} \\
 {} & {} & Q \times \two
\end{tikzcd}\]
We refer to $\rlp{\Gamma}$ as the class of \emph{naive fibrations}.

\begin{remark}
In view of \cite[Remarque~1.3.15]{Cisinski2006}, the set $\Gamma$ generates the smallest class $\An$ of \emph{anodyne extensions} containing $\emptyset$, and so 
\[
\rlp{\Gamma} = \rlp{(\llp{(\rlp{\Gamma})})} = \rlp{\An}
\]
coincides with the class of naive fibrations in the sense of \cite[Definition~2.4.18]{Cisinski2019}.
\end{remark}
It follows from \cite[Theorem~2.4.19]{Cisinski2019} that the presheaf category $\w{\P_k}$ admits a (Cisinski) model structure in which:
\begin{enumerate}[label=(\roman*)]
\item The cofibrations are precisely the monomorphisms (and they coincide with the class $\llp(\rlp{\Pi})$ by Lemma~\ref{l:Pi-generator}).
\item A presheaf~$X\in \w{\P_k}$ is fibrant just when the unique arrow $!\colon X \to \one$ is a naive fibration.
\end{enumerate}
Note that fibrant objects in Proposition~\ref{p:model-structure} are defined as those presheaves~$X$ such that the unique morphism $!\colon X\to \one$ belongs to $\rlp{\Lambda}$, where $\Lambda$ is the subclass of $\Gamma$ defined by setting $e = 1$. As $\rlp{\Lambda} = \rlp{\Gamma}$, the two notions of fibrant object are one and the same. Hence, Proposition~\ref{p:model-structure} follows.

%

\bibliographystyle{amsplain-nodash}
\providecommand{\bysame}{\leavevmode\hbox to3em{\hrulefill}\thinspace}
\providecommand{\MR}{\relax\ifhmode\unskip\space\fi MR }
\providecommand{\MRhref}[2]{%
  \href{http://www.ams.org/mathscinet-getitem?mr=#1}{#2}
}
\providecommand{\href}[2]{#2}

\end{document}